\documentclass[12pt]{amsart}

\usepackage{fullpage,url,amssymb,amsmath,amsthm,amsfonts,mathrsfs,hyperref,enumitem}
\usepackage[alphabetic,lite,nobysame]{amsrefs}
\usepackage[all, cmtip]{xy} 
\usepackage{ulem}

\usepackage{color} 

\DeclareFontEncoding{OT2}{}{} 
\newcommand{\textcyr}[1]{{\fontencoding{OT2}\fontfamily{wncyr}\fontseries{m}\fontshape{n}\selectfont #1}}
\newcommand{\Sha}{{\mbox{\textcyr{Sh}}}}

\newcommand{\Z}{{\mathbb Z}}
\newcommand{\Q}{{\mathbb Q}}

\newcommand{\F}{{\mathbb F}}

\newcommand{\calO}{\mathcal{O}}

\newcommand{\defi}[1]{\textsf{#1}} 

\DeclareMathOperator{\End}{End}
\DeclareMathOperator{\res}{res}
\DeclareMathOperator{\Gal}{Gal}
\DeclareMathOperator{\HH}{H}
\DeclareMathOperator{\Aut}{Aut}
\DeclareMathOperator{\GL}{GL}

\newtheorem{Theorem}{Theorem}[section]
\newtheorem{Lemma}[Theorem]{Lemma}
\newtheorem{Proposition}[Theorem]{Proposition}
\newtheorem{Corollary}[Theorem]{Corollary}
\newtheorem{Definition}[Theorem]{Definition}

\newtheorem{Remark}[Theorem]{Remark}

\numberwithin{equation}{section}
\numberwithin{table}{section}

\begin{document}

\title{The local-global principle for divisibility in CM elliptic curves}

\author{Brendan Creutz}

\author{Sheng(Victor) Lu}

\begin{abstract}
	We consider the local-global principle for divisibility in the Mordell-Weil group of a CM elliptic curve defined over a number field. For each prime $p$ we give sharp lower bounds on the degree $d$ of a number field over which there exists a CM elliptic curve which gives a counterexample to the local-global principle for divisibility by a power of $p$. As a corollary we deduce that there are at most finitely many elliptic curves (with or without CM) which are counterexamples with $p > 2d+1$. We also deduce that the local-global principle for divisibility by powers of $7$ holds over quadratic fields.
\end{abstract}

\maketitle

\section{Introduction}%

	Let $E/k$ be an elliptic curve over a number field $k$. We say that a $k$-rational point $P \in E(k)$ is divisible by the integer $N$ if there exists $Q \in E(k)$ such that $NQ = P$. The question motivating this paper is the extent to which this notion of divisibility satisfies a local-global principle. Namely, if there exists $Q_v \in E(k_v)$ such that $NQ_v = P$ for all (or all but possibly finitely many) completions $k_v$ of $k$ does it follow that $P$ is divisible by $N$? 
	
	Over the past decades there has been substantial interest in the problem of determining conditions on $N$ and $k$ implying that such a local-global principle holds for all elliptic curves over $k$ \cite{DZ1,DZ2,DZ3,PRV-2,PRV-3,Creutz2,LawsonWuthrich,Ranieri}. Due to its connection with a question of Cassels \cite[Problem 1.3]{Cassels}, the analogous question where $E(k) = \HH^0(k,E)$ is replaced by the Galois cohomology group $\HH^1(k,E)$ has also recieved much attention \cite{CipStix,Creutz1,Creutz2}. Function field analogues of these questions were studied in \cite{CreutzVoloch}. In all cases the positive results in the literature concerning local-global divisibility in the groups $E(k)$ and $\HH^1(k,E)$ have relied on the same technique, which considers a more general local-global principle for the $N$-torsion subgroup of $E$ (see~Definition~\ref{def:sha1} below).	
	
	The approach to establishing such a local-global principle can be summarized as follows. First one aims to identify purely group-theoretic conditions on the image of the mod $N$ Galois representation $\rho_N : \Gal(k) \to \Aut(E[N]) \simeq \GL_2(\Z/N)$ which guarantee that the local-global principle for divisibility by $N$ holds. Elliptic curves for which these conditions are not satisfied correspond to non-cuspidal $k$-rational points on some modular curve with level $N$ structure. These curves have only finitely many points defined over number fields of degree $\le d$, provided $N$ is sufficiently large. In many cases one can show that all of the low degree points are cusps. This has resulted in proofs that these local-global principles for divisibility by a prime power $N = p^n$ hold for all $p$ larger than an explicit bound depending only on the degree of the number field (See \cite[Corollary 2]{PRV-2} or \cite[Theorem B(1)]{CipStix}). In the case $k = \Q$, the bound is $p \ge 5$ \cite[Corollary 4]{PRV-3} and it is known to be sharp \cite{Creutz2}. For degrees greater than $1$ the exact bound is unknown.
	
	Establishing an exact bound requires identifying the sporadic points on these modular curves and checking whether the local-global principle holds for the corresponding elliptic curves. To that end, we undertake a detailed analysis of the local-global principle for divisibility on CM curves, as these are a common source of low degree points on modular curves.
	
	Before stating our main results let us define the local-global principle we refer to. 
	\begin{Definition}\label{def:sha1}
	For a set of places $S$ of $k$ define
	\begin{equation*}
		\Sha^1(k,E[N];S) := \ker\left(\HH^1(k,E[N]) \to \prod_{v\notin S}\HH^1(k_v,E[N])\right)\,,
	\end{equation*}
	where $\HH^1(k,E[N])$ denotes Galois cohomology of the $N$-torsion subgroup of $E$. We say that the {\defi local-global principle holds for $(E/k,N)$} if $\Sha^1(k,E[N];S) = 0$ for every finite set of places $S$ of $k$.
	\end{Definition}
	If the local-global principle holds for $(E/k,N)$, then the local-global principle for divisibility by $N$ holds for $\HH^i(k,E)$ for all $i \ge 0$ (See \cite[Theorem 2.1]{Creutz2} and Lemma~\ref{lem:H1star}). The goal of this paper is to determine the minimal degree of a number field over which there is a CM elliptic curve for which the local-global principle fails for given prime power $N = p^n$. In Section~\ref{sec:proofs} we prove the following.
	
	\begin{Theorem}\label{thm:MainThm}
		Let $\calO \subset K$ be an order of conductor $f$ in a quadratic imaginary field $K$ and let $j = j(\calO)$ be the $j$-invariant of an elliptic curve with complex multiplication by $\calO$. Let $p^n$ be an odd prime power, let $k = \Q(j)$ and set $u = 2$ if $j \ne 0$ and $u = 3$ if $j = 0$.
			\begin{enumerate}
			\item Let $L$ be a number field and let $E/L$ be an elliptic curve with CM by $\calO$. Then the local-global principle for $(E/L,p^n)$ holds in any of the following cases:
				\begin{enumerate}
					\item $p$ does not divide $f$ and $p$ splits in $K$;
					\item $p$ does not divide $f$, $p$ is inert in $K$ and $[L:k] < (p^2-1)/u$; or
					\item $p$ divides $f$ or $p$ ramifies in $K$ and $[L:k] < (p-1)/2$.
				\end{enumerate}
			\item These bounds above are sharp:
				\begin{enumerate}
					\item[(b')] If $p$ does not divide $f$ and $p$ is inert in $K$, then there exists a number field $L$ of degree $(p^2-1)/u$ over $k$ and an elliptic curve $E/L$ with $j(E) = j$ such that the local-global principle fails for $(E/L,p^2)$.
					\item[(c')] If $p$ ramifies in $K$ but does not divide $f$, then there exists a number field $L$ of degree $(p-1)/2$ over $k$ and an elliptic curve $E/L$ with $j(E) = j$ such that the local-global principle fails for $(E/L,p^2)$.
				\end{enumerate}
			\end{enumerate}
	\end{Theorem}

		Using Theorem~\ref{thm:MainThm} one can determine the minimal degree of a number field $L$ for which there exists a CM elliptic curve $E/L$ for which the local-global principle for $(E/L,p^n)$ fails for some $n$. In Section~\ref{sec:examples} we give several explicit examples where the local-global principle fails over number fields of minimal degree. The following table gives the values $d = d(p)$ for some small values of $p$.
	\begin{center}
		\begin{tabular}{|c||c|c|c|c|c|c|c|c|c|}
			\hline
			$p$ & $3$ & $5$ & $7$ & $11$ & $13$ & $17$ & $19$ & $23$  \\\hline
			$d$ & $1$ & $4$ & $3$ & $5$ & $12$ & $32$ & $9$ & $33$\\\hline
		\end{tabular}
	\end{center}
	The case $p = 3$ recovers the examples given in \cite{Creutz2,LawsonWuthrich} showing that the local-global principle for $(E/\Q,9)$ can fail. For further details see Section~\ref{sec:p=3}. 

	Combining the above with \cite{Ranieri} and explicit lower bounds for the gonality of modular curves \cite{Abramovich} we will prove the following.

	\begin{Theorem}\label{thm:bounds}
		Let $d\ge 1$ be an integer and let $p \ge 17$ be a prime number $p > 2d + 1$. Then there are at most finitely many elliptic curves $E/L$ defined over a number field of degree $d = [L:\Q]$ such that the local-global principle for $(E/L,p^n)$ fails for some $n \ge 1$. Moreover, any such counterexample to the local-global principle yields a non-cuspidal non-CM point of degree $\le d$ on the modular curve $X(p)$ parameterizing isomorphism classes of elliptic curves with full level $p$ structure $E[p] \simeq \mu_p \times \Z/p$.
	\end{Theorem}
	
	Theorem~\ref{thm:bounds} should be compared with \cite[Corollary 2]{PRV-2} and \cite[Theorem B(1)]{CipStix} which assert that the local-global principle holds for $(E/L,p^n)$ for all $[L:\Q] \le d$ provided $p > (1+3^{d/2})^2$.  The conclusion of our corollary is weaker in that it allows finitely many possible exceptions, but our bound on $p$ is linear rather than exponential in the degree $d$. We expect that our bound holds without exceptions for most (if not all) primes $p$. Sporadic points of degree $d \le (p-1)/2$ on $X(p)$ should be quite rare as these curves have gonality $\Theta(p^3)$. Moreover, the existence of such a point does not necessarily imply that there is a counterexample to the local-global principle, as there are additional (and rather strict) conditions which must also be satisfied by the mod $p^2$ Galois representation of the corresponding elliptic curves. 
	
	The points of degree at most $2$ on the Klein quartic $X(7)$ are determined in \cite{Tzermias}. The rational points are all cusps and the degree $2$ points have residue field $\Q(\sqrt{-3})$ and lie above $j = 0$ on $X(1)$. Since $7$ splits in $\Q(\sqrt{-3})$, Theorem~\ref{thm:MainThm} shows that the local-global principle for $(E/\Q(\sqrt{-3}),7^n)$ holds for the corresponding curves. Thus the following corollary.
	
	\begin{Corollary}\label{cor:7}
		The local-global principle holds for $(E/L,7^n)$ for every elliptic curve $E/L$ over a quadratic number field and every $n \ge 1$.
	\end{Corollary}
	
	Note that by Theorem~\ref{thm:MainThm} the local-global principle with $N=7^n$ can fail for elliptic curves over cubic number fields. For an explicit example, see Section~\ref{sec:p=7}.

	\section{Group theoretic results on $\HH^1_*$}
		
		Let $p$ be an odd prime.
		
		\begin{Definition}
			Let $V_n := \Z/p^n \times \Z/p^n$ be the natural module with a left action of $\GL_2(\Z/p^n)$. For a subgroup $G \subset \GL_2(\Z/p^n)$ let $\HH^i(G,V_n)$ denote the $i$-th cohomology group of the $G$-module $V_n$. Define 
			\[
				\HH^1_*(G,V_n) := \bigcap_{g \in G} \ker\left( \HH^1(G,V_n) \stackrel{\res_g}\to \HH^1(\langle g \rangle,V_n)\right)\,,
			\]
			where $\langle g \rangle$ denotes the cyclic subgroup of $G$ generated by $g$.
		\end{Definition}
		
		The following lemma is well known in the literature on questions of local-global divisibility.
		
		\begin{Lemma}\label{lem:H1star}
			Let $E/k$ be an elliptic curve over a number field and let $G \subset \GL_2(\Z/p^n)$ denote the image of the representation $\Gal(k) \to \Aut(E[n]) \simeq \GL_2(\Z/p^n)$ (for some choice of isomorphism $\Aut(E[p^n]) \simeq \GL_2(\Z/p^n)$). Then the local global principle holds for $(E/k,p^n)$ if and only if $\HH_*^1(G,V_n) = 0$.
		\end{Lemma}

		\begin{proof}
			To simplify notation let $\mathbb{K} := k(E[p^n])$ and identify $G \simeq \Gal(\mathbb{K}/k)$. For each place $v$ of $k$, choose a place $\frak{v}$ of $\mathbb{K}$ above $v$ and let $G_{\frak{v}} = \Gal(\mathbb{K}_{\frak{v}}/k_v)$ be the decomposition group. For any finite set of primes $S$, the inflation-restriction sequence gives the following commutative and exact diagram.
			\[
				\xymatrix{
					0 \ar[r]& \HH^1(G,E[p^n]) \ar[r]^\inf\ar[d]^a & \HH^1(k,E[p^n]) \ar[r]^\res \ar[d]^b & \HH^1(\mathbb{K},E[p^n]) \ar[d]^c\\
					0 \ar[r]& \prod_{v\notin S}\HH^1(G_v,E[p^n]) \ar[r] & \prod_{v\notin S}\HH^1(k_v,E[p^n]) \ar[r] & \prod_{v\notin S} \HH^1(\mathbb{K}_{\frak{v}},E[p^n]) 
				}
			\]
			Since $\HH^1(\mathbb{K},E[p^n]) = \operatorname{Hom}_{\textup{cont}}(\Gal(\mathbb{K}),E[p^n])$, Chebotarev's density theorem implies that the map $c$ is injective. Hence $\Sha^1(k,E[p^n];S) = \ker(b) = \inf(\ker(a))$. By a second application of Chebotarev's density theorem, the groups $G_v$ range (up to conjugacy) over all cyclic subgroups of $G$. From this it follows that $\ker(a) \subseteq \HH^1_*(G,E[p^n])$. We deduce from this that $\Sha^1(k,E[p^n];S)  \subset \inf\left(\HH^1_*(G,E[p^n])\right)$ with equality in the case that $S$ contains all of the finitely many places where the decomposition group is not cyclic. The result follows.
		\end{proof}

		\begin{Definition}\label{def:Cdel}
			For an odd integer $m \ge 3$ and $\delta \in \Z/N$ define
		\begin{align*}
			C_{\delta,m} &:= \left\{ \left[\begin{matrix} a & b \\ \delta b & a \end{matrix}\right] \;:\; a, b \in \Z/m,\, a^2 - \delta b^2 \in (\Z/m)^\times \right\} \subset \GL_2(\Z/m)\,, \text{ and}\\			
			N_{\delta,m} &:= \left\langle \left[\begin{matrix} - 1 & 0 \\ 0 & 1 \end{matrix}\right], C_{\delta,m} \right\rangle \subset \GL_2(\Z/m)\,.
		\end{align*}
		When $m = p^n$ is a prime power, we say that $G \subset N_{\delta,p^n}$ is a {\bf full subgroup} if the kernels of the reduction mod $p$ maps $N_{\delta,p^n} \to \GL_2(\Z/p)$ and $G \to \GL_2(\Z/p)$ are equal.
		\end{Definition}
		
		\begin{Lemma}\label{lem:reducetoC}
			Let $G \subset N_{\delta,p^n}$ and let $G' := G \cap C_{\delta,p^n}$. If $\HH^1_*(G,V_n) \ne 0$, then $\HH^1_*(G',V_n) \ne 0$.
		\end{Lemma}
		
		\begin{proof}
			Note that $G'$ has odd order and index dividing $2$ in $G$. So $\HH^i(G/G',V_n^{G'}) = 0$ for $i \ge 1$. Thus, the inflation-restriction sequence gives an injective map $ \HH^1(G,V_n) \to \HH^1(G',V_n)$. This map sends $\HH^1_*(G,V_n)$ to $\HH^1_*(G',V_n)$ because every cyclic subgroup of $G'$ is also a cyclic subgroup of $G$.
		\end{proof}

		\subsection{Split case}
			\begin{Lemma}\label{lem:splitcase}
				Suppose $\delta$ is a nonzero square mod $p$. Then for every $G \subset N_{\delta,p^n}$, we have $\HH^1_*(G,V_n) = 0$.
			\end{Lemma}
			
			\begin{proof}
				By Lemma~\ref{lem:reducetoC} we may assume that $G \subset C_{\delta,p^n}$. Let $d \in \Z/p^n$ be a square root of $\delta$. Then $C_{\delta,p^n}$ is conjugate to the group of diagonal matrices in $\GL_2(\Z/p^n)$. Since $G$ is diagonal, $V_n$ splits as a product $V_n = W_1\times W_2$ of cylic $G$-modules of order $p^n$. Hence $\HH^1(G,V_n) \simeq \HH^1(G,W_1)\times\HH^1(G,W_2)$. We will show below that $\HH^1_*(G,W_i) = 0$ for $i = 1,2$. It follows that $\HH^1_*(G,V_n) = 0$ as required. 
				
				Write $G = H_1 \times H_2$ where $H_i \subset G$ is the subgroup containing all matrices whose $i$-th diagonal entry is $1$. Note that $W_1^{H_1} = W_1$ and that $H_2$ acts faitfully on $W_1$ (i.e., through an injective map $H_2 \to \Aut(W_2) \simeq (\Z/p^n)^\times$). It follows from a standard computation in the cohomology of cyclic groups that $\HH^1(H_2,W_1) = 0$ (see \cite[Lemma 9.1.4]{CoNF}). Let $\xi \in \HH^1_*(G,W_1)$. Since $H_1 \subset G$ is a cyclic subgroup the restriction of $\xi$ to $H_1$ is trivial. Hence $\xi$ is in the image of the inflation map $\HH^1(H_2,W_1) = \HH^1(H_2,W_1^{H_1}) \to \HH^1(G,W_1)$. As noted above, $\HH^1(H_2,W_1) = 0$, so $\xi = 0$ showing that $\HH^1_*(G,W_1) = 0$. Swapping indices the same argument shows that $\HH^1_*(G,W_2)=0$.
			\end{proof}
			
		\subsection{Inert case}
		\begin{Lemma}\label{lem:inertcase}
			Suppose that $\delta$ is not a square modulo $p$. Let $G \subset N_{\delta,p^n}$ and let $G_1 \subset \GL_2(\Z/p)$ denote the image of $G$ modulo $p$.
			\begin{enumerate}
				\item If $G_1$ is contained in neither $\left[\begin{matrix} \pm 1 & 0 \\ 0 & 1 \end{matrix}\right]$ nor $\left[\begin{matrix} 1 & 0 \\ 0 & \pm 1 \end{matrix}\right]$, then $\HH^1_*(G,V_n) = 0$.
				\item If $G$ is a full subgroup of $N_{\delta,p^2}$ with $G_1 \subset \left[\begin{matrix} \pm 1 & 0 \\ 0 & 1 \end{matrix}\right]$ or $G_1 \subset \left[\begin{matrix} 1 & 0 \\ 0 & \pm 1 \end{matrix}\right]$, then $\HH^1_*(G,V_2) \ne 0$.
			\end{enumerate}
		\end{Lemma}
		
		\begin{proof}
			Let us prove the first statement. Suppose $\HH^1_*(G,V_n) \ne 0$. Letting $G' = G \cap C_{\delta,p^n}$ we have $\HH^1_*(G',V_n) \ne 0$ by Lemma~\ref{lem:reducetoC}. Let $G'_1$ denote the image of $G'$ modulo $p$. Since $\#C_{\delta,p} = p^2 - 1$ is prime to $p$, \cite[Theorem 2]{Ranieri} there are two possibilities for $G'_1$
			\begin{enumerate}
				\item[(a)] $G_1'$ is generated by a element of order dividing $p-1$ with $1$ as an eigenvalue, or 
				\item[(b)]$G_1'$ is generated by an element of order $3$ acting irreducibly on $V_1 = p^{n-1}V_n$.
			\end{enumerate}
			(We note that $G_1' = S_3$ is impossible because $C_{\delta,p}$ is abelian). First consider case (a). The elements of order $p-1$ in $C_{\delta,p}$ are diagonal matrices, so the condition on the eigenvalues implies that $G'_1$ is trivial. Then $G_1$ is generated by an element of order dividing $2$ which has $1$ as eigenvalue, so it must be contained in one of the two groups in the statement. 
			
			Now consider case (b). Then $G'$ is abelian of order $3p^m$, so the Sylow-$3$-subgroup $P \subset G'$ is normal. The inflation-restriction sequence reads
			\[
				\HH^1(G'/P,V_n^P) \to \HH^1(G',V_n) \to \HH^1(P,V_n)\,.
			\]
			Since $P$ acts irreducibly on $V[p]$ we have $V_n^P = 0$, so the first term in the sequence is $0$. The final term in the sequence is also trivial because $P$ and $V_n$ have relatively prime orders. By exactness of the inflation-restriction sequence we conclude $\HH^1(G',V_n) = 0$, contradicting the assumption $\HH^1_*(G',V_n) \ne 0$.
			
			We now prove part $2$ of the lemma. Consider the matrices
			\[
				\sigma_1 = \left[\begin{matrix} -1 & 0 \\ 0 & 1 \end{matrix}\right]\,, \sigma_2 = \left[\begin{matrix} 1 & 0 \\ 0 & -1 \end{matrix}\right]\,, h_1 = \left[\begin{matrix} 1+p & 0 \\ 0 & 1+p \end{matrix}\right],\, h_2 = \left[\begin{matrix} 1 & p \\ \delta p & 1 \end{matrix}\right] \in \GL_2(\Z/p^2)\,.
			\]
			By assumption $G$ is generated by $h_1,h_2$ and at most one of the $\sigma_i$. Then $G$ is the semidirect product of $H = \langle h_1,h_2\rangle$ and a subgroup of order dividing $2$. Since $G/H$ has order dividing $2$ and $p$ is odd the inflation-restriction sequence gives an isomorphism
			\[
				\HH^1(G,V_2[p]) \simeq \HH^1(H,V_2[p])^{G/H} = \operatorname{Hom}_{G/H}(H,V_2[p])\,.
			\]
			Let ${\bf v} \in V_2[p]^G$ be a nonzero element fixed by $G$ and define $\phi : H \to V_2[p]$ as the homomorphism determined by $\phi(h_1) = {\bf v}$ and $\phi(h_2) = 0$. Since $h_1$ lies in the center of $G$ and ${\bf v}$ is fixed by $G$, $\phi$ is a $G/H$-equivariant homomorphism. By the isomorphism above this determines a nonzero class in $\HH^1(G,V_2[p])$. We claim that the image $\phi'$ of this class in $\HH^1(G,V_2)$ is a nonzero element of $\HH^1_*(G,V_2)$.
			
			Let us give the details assuming $\sigma_1 \in G$, the other cases being handled similarly. Let $g \in G$. We will show that the restriction of $\phi$ to the subgroup generated by $g$ is a coboundary. If $g \in H$, then $g = h_1^ah_2^b$ for some $a,b$ and the condition that $\phi'$ restricts to a coboundary on the subgroup generated by $g$ is that the equation
			\begin{equation}\label{matrixeq}
				\left[\begin{matrix} ap & bp \\ b\delta p & ap \end{matrix}\right]{\bf x} = \left[\begin{matrix} 0 \\ ap \end{matrix}\right]
			\end{equation}
			has a solution ${\bf x} \in V_2$. This clearly has solutions when $ap = 0$. When $ap \ne 0$, $(a^2-\delta b^2) \in (\Z/p^2)^\times$ because we have assumed $\delta$ is not a square modulo $p$. In this case the unique solution to~\eqref{matrixeq} is ${\bf x} = \frac{a}{a^2-\delta b^2}\left[\begin{matrix} -b \\ a \end{matrix}\right]$. 
			
			If, on the other hand, $g \not\in H$, then $g = \sigma h_1^ah_2^b$ in which case the local condition becomes
			\begin{equation}\label{matrixeq2}
				\left[\begin{matrix} ap & bp \\ -b\delta p & -2-ap \end{matrix}\right]{\bf x} = \left[\begin{matrix} 0 \\ ap \end{matrix}\right]\,,
			\end{equation}
			which has the solution $x = p$, $y = -ap/2$.
			
			The fact that~\eqref{matrixeq} and~\eqref{matrixeq2} have solutions for any choice of $a,b$ gives that $\phi' \in \HH^1_*(G,V_2)$. The fact that there is no common solution to~\eqref{matrixeq} as one varies $a,b$ shows that $\phi'$ is not trivial.	
		\end{proof}
		
		\subsection{Ramified case}
		\begin{Lemma}\label{lem:ramifiedcase}
			Suppose that $\delta \equiv 0$ mod $p$. Let $G \subset N_{\delta,p^n}$ and let $G_1$ denote the image of $G$ modulo $p$.
			\begin{enumerate}
				\item If $G_1$ is contained in neither $\left[\begin{matrix} 1 & * \\ 0 & \pm 1 \end{matrix}\right]$ nor $\left[\begin{matrix} \pm1 & 0 \\ 0 & 1 \end{matrix}\right]$, then $\HH^1_*(G,V_n) = 0$.
				\item If $\delta \ne 0 \bmod p^2$, $G$ is a full subgroup of $N_{\delta,p^2}$ and $G_1 = \left[\begin{matrix} 1 & * \\ 0 & \pm 1 \end{matrix}\right]$, then $\HH^1_*(G,V_2) \ne 0$.
			\end{enumerate}
		\end{Lemma}
		
		\begin{proof}
			For the first statement suppose $\HH^1_*(G',V_n) \ne 0$ where $G' = G \cap C_{\delta,p^n}$. Let $G'_1$ denote the image of $G'$ modulo $p$. If $p \nmid \#G_1'$, then as in the proof of the preceding lemma, \cite{Ranieri} implies that $G_1$ is generated by a diagonal matrix of order dividing $2$ with $1$ as an eigenvalue. Otherwise $p \mid \#G_1'$. Since $\delta \equiv 0 \mod p$, $C_{\delta,p}$ is a Borel subgroup. So in this case \cite{Ranieri} implies that $G_1'$ is the subgroup of stricly upper triangular matrices and that $G_1 = G_1'$ or $G_1$ is generated by $G_1'$ and $\operatorname{diag}(1,-1)$ as required.
			
		The assumption in the second statement of the lemma implies that $G$ is generated by the matrices
		\[
			\sigma = \left[\begin{matrix} 1 & 0 \\ 0 & -1 \end{matrix}\right]\,,\,g = \left[\begin{matrix} 1 & 1 \\ \delta & 1 \end{matrix}\right]\,,\,h = \left[\begin{matrix} 1+p & 0 \\ 0 & 1+p \end{matrix}\right] \in \GL_2(\Z/p^2)\,.
		\]
		We note that any element of $G$ can be written in the form $\sigma^ag^bh^c$ for some integers $a,b,c$. Let $H = \langle h,g^p \rangle$ be the kernel of reduction modulo $p$. Then $G/H$ is the dihedral group of order $2p$ generated by the images $\overline{\sigma}$ and $\overline{g}$ of $\sigma$ and $g$. A direct calculation shows that the cochain defined by
		\[
			\overline{\sigma}^a\overline{g}^b \mapsto p\left[\begin{matrix} b(b-1)/2 \\ (-1)^ab+(1+(-1)^{a+1})/2 \end{matrix}\right]
		\]
		gives a nontrivial class in $\HH^1(G/H,V_2[p])$. We will show that the image $\xi$ of this class in $\HH^1(G,V_2)$ is a nonzero element of $\HH^1_*(G,V_2)$. The proof is similar to that found in~\cite[Lemma 11]{Ranieri}.
		
		By induction one proves that
		\[
			g^b = \left[\begin{matrix} 1+\delta\frac{b(b-1)}{2} & b + \delta\sum_{i = 1}^b \frac{i(i-1)}{2} \\ \delta b & 1+\delta\frac{b(b-1)}{2}   \end{matrix}\right]\,.
		\]
		If $C\subset G$ is a cyclic subgroup generated by $\gamma = g^bh^c$, the condition that $\xi$ is the class of a coboundary on $C$ is that the equation
		\begin{equation}\label{eq3}
			\left[\begin{matrix} cp + \delta\frac{b(b-1)}{2} & b + cp + \delta\sum_{i = 1}^b \frac{i(i-1)}{2} \\ \delta b & cp + \delta\frac{b(b-1)}{2}   \end{matrix}\right]\left[\begin{matrix} x \\ y \end{matrix}\right] = \left[\begin{matrix} p\frac{b(b-1)}{2} \\ pb \end{matrix}\right]\,.
		\end{equation}
		has a solution with $x,y \in \Z/p^2$ for any choice of integers $b,c$. Since the right hand side lies in $pV_2 = V_2[p]$ and the determinant of the matrix on the left hand side is $\delta b^2 \not\equiv 0 \bmod p^2$, this equation has a solution. Namely, $x = p/\delta, y = -cp^2/\delta b$ (which is well defined in $\Z/p^2$ since $\delta \not\equiv 0 \bmod p^2$).
		
		Similarly, if $C$ is generated by $\sigma g^bh^c$, the local condition gives rise to the equation
		\[
			\left[\begin{matrix} cp + \delta\frac{b(b-1)}{2} & b + bcp + \delta\sum_{i = 1}^b \frac{i(i-1)}{2} \\ -\delta b & -2-cp - \delta\frac{b(b-1)}{2}   \end{matrix}\right]\left[\begin{matrix} x \\ y \end{matrix}\right] = \left[\begin{matrix} p\frac{b(b-1)}{2} \\ -pb+p \end{matrix}\right]\,.
		\]
		in which case $x = 0, y = (b-1)p/2$ is a solution. We conclude that $\xi$ lies in $\HH^1_*(G,V_2)$. As the solutions to~\eqref{eq3} depend on $b$, $\xi$ is nontrivial.		
		\end{proof}

	\section{Proofs of the theorems}\label{sec:proofs}
		Before beginning the proof let us recall some relevant results concerning the mod $N$ representations attached to CM elliptic curves.
		
		Let $E/\Q(j(E))$ be an elliptic curve over $k = \Q(j(E))$ with complex multiplication by an order $\calO \subset K$ where $K$ is a quadratic imaginary field. Let $H = K(j(E))$ and let $h : E \to E/\Aut(E) = \mathbb{P}^1$ be a Weber function. All elliptic curves with CM by $\calO$ are twists of one another and the field $H_N := H(h(E[N]))$ does not depend on the choice of twist.		
			
			As $E[N]$ is an $\End(E) = \calO$ module of rank $1$ there is an isomorphism $\Aut_{\calO}(E[N]) \simeq (\calO/N)^\times$.  Assuming $N$ is odd, the natural map $\calO^\times \to (\calO/N)^\times$ is injective and its image identifies with $\Aut(E)$ as a subgroup of $\Aut_{\calO}(E[N])$. The restriction of $\rho_{H,N}$ to $G_{H_N}$ induces a representation $\rho_{H_N} : G_{H_N} \to \Aut(E) \simeq \calO^\times$. In particular, $\Gal(H(E[N])/H_N)$ may be viewed as a subgroup of $\Aut(E)$. On the other hand, any choice of basis for $E[N]$ determines an isomorphism of groups $\Aut(E[N]) \simeq \GL_2(\Z/N\Z)$. The main theorems of class field theory allow one to classify the possibilities for the image of the mod $N$ representation $\rho_{k,N} : \Gal(k) \to \Aut(E[N])$. The following is taken from \cite{Lozano-Robledo}.
			
			\begin{Theorem}[{\cite[Theorem 1.1]{Lozano-Robledo}}]\label{Thm:LozanoRobledo}
				Suppose $N$ is odd and let $\delta = \Delta_Kf^2/4$, where $\Delta_K$ is the fundamental discriminant of $K$ and $f$ is the conductor of $\calO$. Then there is a basis for $E[N]$ such that  the image of $\rho_{k,N} : \Gal(k) \to \GL_2(\Z/N)$ lies in the group $N_{\delta,N}$ (see Definition~\ref{def:Cdel}) and is generated by $\left[\begin{matrix} -1 & 0 \\ 0 & 1 \end{matrix}\right]$ and $C_{\delta,N} = \textup{image}(\rho_{H,N})$. Moreover the index of the image of $\rho_{H,N}$ in $C_{\delta,N}$ is equal to the index of $\Gal(H(E[N])/H_N)$ as a subgroup of $\Aut(E) \simeq \calO^\times$. 
			\end{Theorem}
			
			\begin{Lemma}\label{lem:twists}
				Suppose $N = p$ is an odd prime and the mod $p$ representation attached to $E/k$ surjects onto $N_{\delta,p}$. Let $A \subset \Aut(E) \subset N_{\delta,p}$ and $G \subset N_{\delta,p}$ with $A\cap G = 1$. Let $L \subset k(E[p])$ be the fixed field of the group $AG \subset N_{\delta,p}$. There exists a twist $E'/L$ of $E/L$ by a character $\chi : \Gal(L) \to A \subset \Aut(E)$ such that the mod $p$ image attached to $E'/L$ is equal to $G$ if and only if $G$ is a normal subgroup of $AG$. 
			\end{Lemma}
			
			\begin{Remark}
				The subgroup $A = \mu_2 \subset \Aut(E)$ lies in the center of $N_{\delta,p}$ so in this case $G$ is always normal in $AG$. 
			\end{Remark}
			
			\begin{proof}
				To ease notation let $\mathbb{K} = k(E[p])$ and identify $N_{\delta,p} = \Gal(\mathbb{K}/k)$. If $G$ is a normal subgroup of $AG$, then the extension $\mathbb{K}^G/L$ is Galois with Galois group isomorphic to $A$, which may be identified with a subgroup of $\Aut(E) = \mu_m$. Then there is a character $\chi : \Gal(L) \to \mu_m$ with kernel $\Gal(\mathbb{K}^G)$ whose restriction to $\Gal(\mathbb{K}^A)$ is the inverse of $\rho_{E/\mathbb{K}^A} : \Gal(\mathbb{K}^A) \to A \subset \mu_m$. Let $E'/L$ be the twist of $E/L$ by $\chi$. The mod $p$ representations are related by $\rho_{E/L,p} \otimes \chi = \rho_{E'/L,p}$. So $\mathbb{K}^A = \ker(\rho_{E'/L,p}) = L(E'[p])$ and the image of $\rho_{E'/L,p}$ is equal to $G$.

				Conversely, if there exists a twist as in the statement, then $M = \ker(\chi)$ is a Galois extension of $L$ and $\Gal(\mathbb{K}/M) = G$, so $G$ is normal in $AG = \Gal(\mathbb{K}/L)$.
			\end{proof}	

		\begin{proof}[Proof of Theorem~\ref{thm:MainThm}]\hfill
		
		\noindent{\bf Part (1):} Let $L/k$ be a finite extension and let $E/L$ be an elliptic curve with CM by $\calO$. By \cite[Theorem 4.6]{Lozano-Robledo} there exists an elliptic curve $E'/k$ with CM by $\calO$ such that (under a suitable choice of basis for $E'[p^n]$) the image $G_{E/k,p^n}$ of the representation $\rho_{E'/k,p^n} : \Gal(k) \to \Aut(E'[n]) \simeq \GL_2(\Z/p^n)$ is equal to $N_{\delta,p^n}$. The image $G_{E'/L,p^n}$ of the mod $p^n$ representation attached to the base change $E'/L$ is the restriction of $\rho_{E'/k,p^n}$ to the subgroup $\Gal(L) \subset \Gal(k)$. Galois theory gives $[N_{\delta,p^n}:G_{E'/L,p^n}] \le [L:k]$.
		
		There is a character $\chi : \Gal(L) \to \mu_m = \Aut(E)$ such that $E' = E^\chi$ is the twist of $E/L$ by $\chi$. The mod $p^n$ representations are related by $\rho_{E/L,p^n} = \rho_{E'/L,p^n} \otimes \chi$. The images $G_{E/L,p^n}$ and $G_{E'/L,p^n}$ of these representations are subgroups of $N_{\delta,p^n}$ whose sizes differ by a factor which divides $\ell := \#\textup{image}(\chi)$. Thus $[N_{\delta,p^n}:G_{E/L,p^n}] \leq \ell[L:k]$. In particular, if $j \ne 0,1728$, then $[N_{\delta,p^n}:G_{E/L,p^n}] \leq 2[L:k]$. 
		
		\begin{enumerate}
			\item[(a)] Assume that $p$ does not divide $f$ and that $p$ splits in $K$. Then $\delta = \Delta_Kf^2/4$ is a nonzero square modulo $p$. By Lemma~\ref{lem:splitcase} we have $\HH^1_*(G_{E/L,p^n},V_n) = 0$. So the local-global principle holds for $(E/L,p^n)$ by Lemma~\ref{lem:H1star}.
			\item[(b)] Assume that $p$ does not divide $f$, $p$ is inert in $K$ and $[L:k] < (p^2-1)/2$. First assume $j \ne 0,1728$. Then by the discussion above we have 
	\[
		[N_{\delta,p}:G_{E/L,p}] \leq [N_{\delta,p^n}:G_{E/L,p^n}] \leq 2[L:k] < p^2-1\,.	
	\]
	The assumption on $p$ implies that $\delta = \Delta_Kf^2/4$ is not a square modulo $p$.  So $\#N_{\delta,p} = 2(p^2-1)$ and the estimate above gives $\#G_{E/L,p} > 2$. In particular $G_{E/L,p}$ cannot be contained in either of the subgroups appearing in Lemma~\ref{lem:inertcase}. We conclude from this and Lemma~\ref{lem:H1star} that the local-global principle holds for $(E/L,p^n)$.
	
	Now we consider the cases $j = 0$ or $j = 1728$. Let $m = \#\Aut(E) \in \{ 4,6\}$. Suppose $\HH^1_*(G_{E/L,p^n},V_n) \ne 0$. We must show $[L:k] \ge 2(p^2-1)/u$. By Lemma~\ref{lem:inertcase}, $G_{E/L,p}$ is trivial or is generated by $\operatorname{diag}(-1,1)$ or $\operatorname{diag}(1,-1)$. If $G$ is trivial, then the estimate $[N_{\delta,p^n}:G_{E/L,p^n}] \leq u[L:k]$ gives $[L:k] \ge 2(p^2-1)/u$. If $G_{E/L,p}$ is generated by either $\operatorname{diag}(-1,1)$ or $\operatorname{diag}(1,-1)$, then $G_{E/L,p}$ is not normal in $G_{E/L,p}\Aut(E)$. In fact, the only nontrivial subgroup $A \subset \Aut(E)$ for which $G_{E/L,p}$ is normal in $G_{E/L,p}A$ is $A = \mu_2$. By Lemma~\ref{lem:twists} we conclude that the image of $\chi$ is contained in $\mu_2$. So $\ell := \#\textup{image}(\chi) = 2$ and our estimate above gives $[N_{\delta,p^n}:G_{E/L,p^n}] \leq \ell[L:k] \le 2[L:k]$, which implies $[L:k] \ge (p^2-1)/2 \ge 2(p^2-1)/u$ as required.
	
			\item[(c)] Assume that $p$ divides $f$ or $p$ is ramified in $K$. Assume that $[L:k] < (p-1)/2$. These conditions imply $j \ne 0,1728$ (Note that the condition on $[L:k]$ implies $p >3$), so $\Aut(E) = \mu_2$. Arguing as in the previous case we have $[N_{\delta,p}:G_{E/L,p}] < (p-1)/2$. In this case $\delta = \Delta_Kf^2/4$ is $0$ mod $p$, so $\#N_{\delta,p} = 2p(p-1)$ and we conclude $\#G_{E/L,p} > 2p$.  In particular $G_{E/L,p}$ cannot be contained in either of the subgroups appearing in Lemma~\ref{lem:ramifiedcase}. We conclude from this and Lemma~\ref{lem:H1star} that the local-global principle holds for $(E/L,p^n)$.
		\end{enumerate}

\noindent{\bf Part (2):} We now show that the bounds obtained in Part (1) are sharp. Let $E'/k$ be, as above, an elliptic curve such that the mod $p^n$ representation surjects onto $N_{\delta,p^n}$. Let $\mathbb{K} = k(E'[p])$. We identify $N_{\delta,p} = \Gal(\mathbb{K}/k)$. Let $H_{p} \subset \mathbb{K}$ be the subfield fixed by $\Aut(E') \subset N_{\delta,p}$. As the notation suggests, $H_{p} = k(h(E'[p]))$ for a Weber function $h$. Hence $H_{p}$ is independent the choice of twist of $E'$. Let $H_p' \subset \mathbb{K}$ be the subfield fixed by $\mu_2 \subset \Aut(E) \subset N_{\delta,p}$. Then $H_p = H_p'$ if $j \ne 0, 1728$.

	\begin{enumerate}
		\item[(b')] Assume that $p$ does not divide $f$ and $p$ is inert in $K$. Let $M \subset \mathbb{K}$ be the subfield fixed by $g = \operatorname{diag}(-1,1) \in N_{\delta,p}$ and let $L = M \cap H_{p}' = \mathbb{K}^{\langle g,-1\rangle}$. Note that $[\mathbb{K}:L] = 4$ and $[\mathbb{K}:k] = \#N_{\delta,p} = 2(p^2-1)$, so $[L:k] = (p^2-1)/2$. Let $\chi : \Gal(L) \to \mu_2$ be the quadratic character with $\ker(\chi) = \Gal(M)$ and let $E/L$ be the quadratic twist of $E'/L$ by $\chi$. The image $G_{E/L,p^2}$ of the mod $p^2$ representation attached to $E/L$ is a full subgroup of $N_{\delta,p^2}$ whose image mod $p$ is generated by $g = \operatorname{diag}(-1,1)$. So by Lemma~\ref{lem:inertcase} we have that $\HH^1_*(G_{E/L,p^2},V_2) \ne 0$. By Lemma~\ref{lem:H1star} we conclude that the local-global principle fails for $(E/L,p^2)$.
		
		In the case $j = 0$ we can construct an example over a field of degree $(p^2-1)/3$ as follows. The field $H_p = k(E'[p])$ has degree $2(p^2-1)/6 = (p^2-1)/3$ and $\Gal(\mathbb{K}/H_p) = \Aut(E)=\mu_6$. By Lemma~\ref{lem:twists} (applied with $G = 1$) there exists a sextic twist of $E'/H_p$ such that $H_p = H_p(E'[p])$. The image of the mod $p^2$ associated to this curve is the full subgroup of $N_{\delta,p^2}$ congruent to the trivial group modulo $p$. By Lemmas~\ref{lem:inertcase} and~\ref{lem:H1star} the local-global principle fails for $(E'/H_p,p^2)$.
		
		 \item[(c')] Assume that $p$ ramifies in $K$ but does not divide $f$. Let $G \subset N_{\delta,p}$ be the subgroup generated by $\operatorname{diag}(1,-1)$ and the strictly upper triangular matrices. Let $M \subset \mathbb{K}$ be the fixed field of $G$ and let $L = M \cap H_p'$. In this case $[\mathbb{K}:k] = 2p(p-1)$, so $[L:k] = (p-1)/2$. 
		 
		As in the preceding case, twisting $E'/L$ by the quadratic character $\Gal(L) \to \mu_2$ with kernel $\Gal(M)$ yields an elliptic curve $E/L$ such that the image $G_{E/L,p^2}$ of the mod $p^2$ representation is the full subgroup of $N_{\delta,p}$ whose image mod $p$ is $G$. By Lemma~\ref{lem:ramifiedcase} and~\ref{lem:H1star} we conclude that the local-global principle fails for $(E/L,p^2)$.
	\end{enumerate}
	\end{proof}
	
		\begin{proof}[Proof of Theorem~\ref{thm:bounds}]
		Let $E/L$ be an elliptic curve over the number field $L$ of degree $d = [L:\Q]$. Suppose $p > 2d+1$ and that the local-global principle for $(E/L,p^n)$ fails. The determinant of the mod $p$ representation $\Gal(k) \to \Aut(E[p]) \simeq \GL_2(\Z/p) \to \Z/p^\times$ is the $p$-cyclotomic character. Since $d < (p-1)/2 = [\Q(\mu_p)^+:\Q]$, the image of this determinant map is of size $>2$. \cite[Theorem 2]{Ranieri} shows that the possibilities for the image of the mod $p$ representation are rather limited. The only possibility compatible with the image of the determinant map having size greater than $2$ is that the image is contained in $\left[\begin{matrix} 1 & 0 \\0 & *\end{matrix}\right]$. In other words, $E[p] \simeq \Z/p\times \mu_p$ as a Galois module, so $E/L$ corresponds to a non-cuspidal point in $X(p)(L)$. By Theorem~\ref{thm:MainThm}, $E/L$ does not have CM. It remains only to prove the finiteness of the set of points of degree $\le (p-1)/2$ on $X(p)$. By \cite{Frey} it suffies to check that $X(p)$ has gonality $\gamma(X(p)) \ge (p-1)$. In \cite{Abramovich} one finds the estimate $\gamma(X(p)) \ge [\operatorname{PSL}_2(\Z):\Gamma(p)]7/800 = 7(p^3-p)/1600$, which suffices for $p \ge 17$.
	\end{proof}

	\section{Explicit Examples}\label{sec:examples}
		
		\begin{Proposition}\label{prop:3mod4}
			Suppose $p \equiv 3 \bmod 4$ is a prime ramifying in $K$ and let $E/k$ be an elliptic curve with CM by an order in $K$ whose conductor is not divisible by $p$. We assume $k = \Q(j(E))$. Then $[k(\mu_p):k] = p-1$. Let $k(\mu_p)^+$ be the unique intermediate field of degree $(p-1)/2$ over $k$. There is a twist $E'/k$ of $E/k$ such that the local-global principle fails for $(E'/k(\mu_p)^+,p^2)$.
		\end{Proposition}
	
		\begin{proof}
			Twisting if necessary, we may assume that the mod $p^n$ representations attached to $E/k$ surject onto $N_{\delta,p^n}$. Let $\mathbb{K} = k(E[p])$ and identify $N_{\delta,p} = \Gal(\mathbb{K}/k)$. Since $\delta \equiv 0 \bmod p$, $N_{\delta,p}$ consists of the upper triangular invertible matrices. Note that $k(\mu_p) \subset \mathbb{K}$ is the subfield fixed by $\operatorname{SL}_2(\Z/p) \cap N_{\delta,p}$. Since $p \equiv 3 \bmod 4$, $\operatorname{SL}_2(\Z/p) \cap N_{\delta,p}$ is the group generated by $-1$ and the strictly upper triangular matrices. The subfield $k(\mu_p)^+$ is fixed by the complex conjugation, which acts on $E[p]$ as $\operatorname{diag}(-1,1)$ or $\operatorname{diag}(1,-1)$. So $k(\mu_p)^+$ is the fixed field of the group $\left[\begin{matrix}\pm1 & * \\ 0 & \pm 1\end{matrix}\right]$ of order $4p$. The fixed field $M \subset \mathbb{K}$ of the group $G = \left[\begin{matrix} 1 & * \\ 0 & \pm 1\end{matrix}\right]$ is a quadratic extension of $k(\mu_p)^+$. Let $E'/k(\mu_p)^+$ be the twist of $E/k(\mu_p)^+$ by the quadratic character with kernel $\Gal(M)$. Then the image of the mod $p$ representation attached to $E'/k(\mu_p)$ is equal to $G$. By Lemma~\ref{lem:ramifiedcase} we have $\HH^1_*(G_{E'/k,p^2},E[p^2]) \ne 0$. The only primes that ramify in the extension $k(\mu_p)^+/k$ are those lying above $p$.
		\end{proof}
		
		\subsection{The case $p = 3$}\label{sec:p=3}
			Proposition~\ref{prop:3mod4} shows that there is an elliptic curve $E/\Q$ of $j$-invariant $0$ (so $K = \Q(\sqrt{-3}$)) such that the local-global principle fails for $(E/\Q,9)$. Examples of such were first given in \cite{Creutz2} and then in \cite{LawsonWuthrich}. In fact the proposition recovers these examples as all have $j$-invariant $0$ and mod $9$ image the full subgroup of $N_{6,9}$ congruent to $\left[\begin{matrix} 1 & * \\ 0 & \pm 1\end{matrix}\right]$ modulo $3$. In light of the fact that $\Aut(E) \simeq \mu_6$ for these curves, one can obtain infinitely many counterexamples to the local-global principle for $(E/\Q,9)$ by taking cubic twists (which was already evident from \cite[Corollary 4.3]{Creutz2}). This family of twists also contains the modular curve $X_0(27)$ whose mod $9$ image is the full subgroup of $N_{6,9}$ congruent to $\left[\begin{matrix} 1 & 0 \\ 0 & \pm 1\end{matrix}\right]$ modulo $3$, giving a counterexample to the local-global principle for $(E/\Q,9)$ with a different mod $3$ image. 
			
			For an example with a different $j$-invariant one can consider the family of curves $E/\Q$ with $j$-invariant $2^43^35^3$ which have CM by the order of conductor $2$ in $\Q(\sqrt{-3})$. In this case $\Aut(E) \simeq \mu_2$ so there is a unique curve in the family whose mod $9$ representation is the full subgroup of $N_{-3,9}$ congruent to $\left[\begin{matrix} 1 & * \\ 0 & \pm 1\end{matrix}\right]$ modulo $3$; it is the curve with Cremona reference 36.a2 and is a counterexample to the local-global principle for $(E/\Q,9)$.
		
		\subsection{The case $p = 7$}\label{sec:p=7}
			There are two elliptic curves of conductor $49$ over $\Q$ with CM by the maximal order in $\Q(\sqrt{-7})$. One is the modular curve $X_0(49)$ and the other, \cite[\href{https://www.lmfdb.org/EllipticCurve/Q/49/a/2}{Elliptic Curve 49.a2}]{LMFDB}, is its twist by the quadratic character corresponding to $\Q(\sqrt{-7})/\Q$. The images of the mod $7$ representations attached to the base changes of these curves to $\Q(\mu_7)^+$ are 
			\[
				\left[\begin{matrix} \pm 1 & * \\ 0 & 1\end{matrix}\right]\quad\text{and}\quad\left[\begin{matrix} 1 & * \\ 0 & \pm 1\end{matrix}\right]\,,
			\]
			respectively. For both curves the local-global principle for $(E/\Q(\mu_7),49)$ fails, while only for the twist $E$ of $X_0(49)$ does it fail for $(E/\Q(\mu_7)^+,49)$. This gives the unique example of a CM elliptic curve over a cubic number field for which the local-global principle fails with $N = 7^n$. Since the conductor is $49 = 7^2$, the decomposition groups $D_\mathfrak{q} \subset \Gal(\Q(E[49])/\Q(\mu_7)^+)$ are cyclic for all primes $\mathfrak{q} \nmid 7$. Moreover, $7$ is totally ramified in the degree $42$ extension $\Q(E[49])/\Q$ and so if $\frak{p}$ is the prime of $\Q(\mu_7)^+$ lying above $7$, then the restriction map $\HH^1(\Q(\mu_7^+),E[7]) \to \HH^1(\Q(\mu_7)^+_\frak{p},E[7])$ is an isomorphism. We conclude that
			\[
				\Sha^1(\Q(\mu_7)^+,E[49];S) \ne 0 \quad \Leftrightarrow \quad \frak{p} \in S\,.
			\]
			So while the local-global principle fails for $(E/\Q(\mu_7)^+,49)$ the local-global principle for divisibility by $7^n$ holds in the groups $E(\Q(\mu_7)^+)$ and $\HH^1(\Q(\mu_7)^+,E)$.
			
		\subsection{The case $p = 5$}\label{sec:p=5}
		
			There is no rational $j$-invariant $j = j(\calO)$ of an order in a quadratic imaginary field such that $5$ divides the conductor or ramifies in $\calO$. So by Theorem~\ref{thm:MainThm} the local-global principle with $N = 5^n$ holds for CM curves over quadratic and cubic fields. The class number of $\Q(\sqrt{-5})$ is $2$, so there are elliptic curves with CM by the maximal order $\calO \subset \Q(\sqrt{-5})$ defined over a quadratic field, namely $k = \Q(\sqrt{5}) = \Q(j(\calO))$. Theorem~\ref{thm:MainThm}(c') implies that there is a CM elliptic curves over some quadratic extension $L/k$ such that the local-global principle for $(E/L,5^2)$ fails. Here we provide an explicit example.
			
			Consider the curve $E/k$ \cite[\href{https://www.lmfdb.org/EllipticCurve/2.2.5.1/4096.1/k/1}{Elliptic Curve 4096.1-k1}]{LMFDB} with Weierstrass equation
			\[
				E : y^2 = f(x) := x^3 - \phi x^2 + (-\phi - 9)x + (-6\phi - 15)\,, 
			\]
			where $\phi \in \Q(\sqrt{5})$ satisfies $\phi^2 + \phi + 1 = 0$. The image of the mod $5$ Galois representation is $\left[\begin{matrix} \pm 1 & * \\ 0 & \pm 1\end{matrix}\right]$ (note that $k = \Q(\mu_5)^+$, so the diagonal entries must be squares in $\F_5^\times$). The $5$-division polynomial of $E/k$ has a root $\theta$ in a quadratic extension $L/k$, which turns out to be $L = \Q(\mu_{20})^+$. The root $\theta$ is the $x$-coordinate of a $5$-torsion point on $E$. The quadratic twist of $E$ by $d = f(\theta) \in L^\times/L^{\times 2}$ yields the curve \cite[\href{https://www.lmfdb.org/EllipticCurve/4.4.2000.1/25.1/a/1}{Elliptic Curves 25.a2}]{LMFDB} which has an $L$-rational $5$-torsion point. The image of the mod $5$ Galois representation attached to $E^d$ is
			\[
				\left[\begin{matrix} 1 & * \\ 0 & \pm 1\end{matrix}\right]
			\]
			and so, by Lemma~\ref{lem:ramifiedcase}, the local-global principle fails for $(E^d/L,25)$. As $5$ is the only prime of bad reduction and $5$ is totally ramified in $L(E^d[5])$ we conclude (similarly to the $p=7$ case) that $\Sha^1(L,E^d[25];S) \ne 0$ if and only if $S$ contains the prime of $L$ above $5$.
			
\section{Acknowledgements}

Some of the content in this article is based on the second author's MSc Thesis at the University of Canterbury which includes a detailed proof of Corollary~\ref{cor:7}.


\begin{bibdiv}
	\begin{biblist}

\bib{Abramovich}{article}{
   author={Abramovich, Dan},
   title={A linear lower bound on the gonality of modular curves},
   journal={Internat. Math. Res. Notices},
   date={1996},
   number={20},
   pages={1005--1011},
   issn={1073-7928},
}

\bib{Cassels}{article}{
   author={Cassels, J. W. S.},
   title={Arithmetic on curves of genus $1$. III. The Tate-\v Safarevi\v c
   and Selmer groups},
   journal={Proc. London Math. Soc. (3)},
   volume={12},
   date={1962},
   pages={259--296},
   issn={0024-6115},
}

	
\bib{CipStix}{article}{
   author={\c{C}iperiani, Mirela},
   author={Stix, Jakob},
   title={Weil-Ch\^{a}telet divisible elements in Tate-Shafarevich groups II: On
   a question of Cassels},
   journal={J. Reine Angew. Math.},
   volume={700},
   date={2015},
   pages={175--207},
   issn={0075-4102},
}

\bib{Creutz1}{article}{
  author={Creutz, Brendan},
   title={Locally trivial torsors that are not Weil-Ch\^atelet divisible},
   journal={Bull. Lond. Math. Soc.},
   volume={45},
   date={2013},
   number={5},
   pages={935--942},
   issn={0024-6093},
}
	
\bib{Creutz2}{article}{
   author={Creutz, Brendan},
   title={On the local-global principle for divisibility in the cohomology
   of elliptic curves},
   journal={Math. Res. Lett.},
   volume={23},
   date={2016},
   number={2},
   pages={377--387},
   issn={1073-2780},
}
		
\bib{CreutzVoloch}{article}{
   author={Creutz, Brendan},
   author={Voloch, Jos\'{e} Felipe},
   title={Local-global principles for Weil-Ch\^{a}telet divisibility in positive
   characteristic},
   journal={Math. Proc. Cambridge Philos. Soc.},
   volume={163},
   date={2017},
   number={2},
   pages={357--367},
   issn={0305-0041},
}
		
		
\bib{DZ1}{article}{
   author={Dvornicich, Roberto},
   author={Zannier, Umberto},
   title={Local-global divisibility of rational points in some commutative
   algebraic groups},
   language={English, with English and French summaries},
   journal={Bull. Soc. Math. France},
   volume={129},
   date={2001},
   number={3},
   pages={317--338},
   issn={0037-9484},
}			

\bib{DZ2}{article}{
   author={Dvornicich, Roberto},
   author={Zannier, Umberto},
   title={An analogue for elliptic curves of the Grunwald-Wang example},
   language={English, with English and French summaries},
   journal={C. R. Math. Acad. Sci. Paris},
   volume={338},
   date={2004},
   number={1},
   pages={47--50},
   issn={1631-073X},
}
				
\bib{DZ3}{article}{
   author={Dvornicich, Roberto},
   author={Zannier, Umberto},
   title={On a local-global principle for the divisibility of a rational
   point by a positive integer},
   journal={Bull. Lond. Math. Soc.},
   volume={39},
   date={2007},
   number={1},
   pages={27--34},
   issn={0024-6093},
}
	
\bib{Frey}{article}{
   author={Frey, Gerhard},
   title={Curves with infinitely many points of fixed degree},
   journal={Israel J. Math.},
   volume={85},
   date={1994},
   number={1-3},
   pages={79--83},
   issn={0021-2172},
}

\bib{Greenberg}{article}{
   author={Greenberg, Ralph},
   title={The image of Galois representations attached to elliptic curves
   with an isogeny},
   journal={Amer. J. Math.},
   volume={134},
   date={2012},
   number={5},
   pages={1167--1196},
   issn={0002-9327},
}

\bib{LMFDB}{misc}{
  label    = {LMFDB},
  author       = {The {LMFDB Collaboration}},
  title        = {The {L}-functions and modular forms database},
  howpublished = {\url{http://www.lmfdb.org}},
  note         = {[Online; accessed 22 October 2021]},
}

\bib{Lozano-Robledo}{article}{
  author={Lozano-Robledo, \'Alvaro},
  title={Galois representations attached to elliptic curves with complex multiplication},
  eprint={arXiv:1809.02584v2}
}
	
\bib{LawsonWuthrich}{article}{
   author={Lawson, Tyler},
   author={Wuthrich, Christian},
   title={Vanishing of some Galois cohomology groups for elliptic curves},
   conference={
      title={Elliptic curves, modular forms and Iwasawa theory},
   },
   book={
      series={Springer Proc. Math. Stat.},
      volume={188},
      publisher={Springer, Cham},
   },
   date={2016},
   pages={373--399},
}
	
	
\bib{CoNF}{book}{
   author={Neukirch, J{\"u}rgen},
   author={Schmidt, Alexander},
   author={Wingberg, Kay},
   title={Cohomology of number fields},
   series={Grundlehren der Mathematischen Wissenschaften [Fundamental
   Principles of Mathematical Sciences]},
   volume={323},
   edition={2},
   publisher={Springer-Verlag},
   place={Berlin},
   date={2008},
   pages={xvi+825},
   isbn={978-3-540-37888-4},
}
	
\bib{PRV-2}{article}{
   author={Paladino, Laura},
   author={Ranieri, Gabriele},
   author={Viada, Evelina},
   title={On local-global divisibility by $p^n$ in elliptic curves},
   journal={Bull. Lond. Math. Soc.},
   volume={44},
   date={2012},
   number={4},
   pages={789--802},
   issn={0024-6093},
}

\bib{PRV-3}{article}{
   author={Paladino, Laura},
   author={Ranieri, Gabriele},
   author={Viada, Evelina},
   title={On the minimal set for counterexamples to the local-global
   principle},
   journal={J. Algebra},
   volume={415},
   date={2014},
   pages={290--304},
   issn={0021-8693},
}

\bib{Ranieri}{article}{
   author={Ranieri, Gabriele},
   title={Counterexamples to the local-global divisibility over elliptic
   curves},
   journal={Ann. Mat. Pura Appl. (4)},
   volume={197},
   date={2018},
   number={4},
   pages={1215--1225},
   issn={0373-3114},
}

\bib{Tzermias}{article}{
   author={Tzermias, Pavlos},
   title={Low-degree points on Hurwitz-Klein curves},
   journal={Trans. Amer. Math. Soc.},
   volume={356},
   date={2004},
   number={3},
   pages={939--951},
   issn={0002-9947},
}
	
	\end{biblist}
\end{bibdiv}
\end{document}